\documentclass[12pt, reqno]{amsart}

\usepackage{amscd}
\usepackage{amsfonts}
\usepackage{bbm}
\usepackage{times}
\usepackage{amssymb}
\usepackage{mathrsfs}
\usepackage{tabularx}
\usepackage{amsthm}
\usepackage{amsmath}
\usepackage{stix2}

\newtheorem*{theorema}{Theorem A}
\newtheorem*{theoremb}{Theorem B}

\newtheorem{thm}{Theorem}
\newtheorem{lem}[thm]{Lemma}
\newtheorem{prop}[thm]{Proposition}
\newtheorem{cor}[thm]{Corollary}

\newtheorem{conj}[thm]{Conjecture}

\newcommand{\C}{{\mathbb C}}
\newcommand{\cn}{{\C^n}}
\newcommand{\bn}{{\mathbb B^n}}

\newcommand{\B}{{\mathcal B}}

\newcommand{\dla}{d\lambda_\alpha}

\newcommand{\inc}{\int_\cn}
\newcommand{\intc}{\int_{\C}}

\DeclareMathOperator{\esssup}{ess\,sup}
\DeclareMathOperator{\re}{Re\,}

\allowdisplaybreaks

\begin{document}
\title{New characterizations for Fock spaces}

\author{Guanlong Bao}
\address[Guanlong Bao]{Department of Mathematics, Shantou University, Shantou, Guangdong 515821, China}
\email{glbao@stu.edu.cn}

\author{Pan Ma}
\address[Pan Ma]{School of Mathematics and Statistics, HNP-LAMA, Central South University,
Changsha, Hunan 410083, China}
\email{pan.ma@csu.edu.cn}

\author{Kehe Zhu}
\address[Kehe Zhu]{Department of Mathematics and Statistics, SUNY,
Albany, NY 12222, USA}
\email{kzhu@albany.edu}

\subjclass[2020]{30H20, 46E15, 46E22.}

\keywords{Fock spaces, Gaussian measure, induced distance, Lipschitz space,
Hardy-Littlewood type theorem.}

\thanks{Ma is supported by NNSF of China (Grant numbers 11801572 and 12171484),
the Natural Science Foundation of Hunan Province (Grant number 2023JJ20056),
the Science and Technology Innovation Program of Hunan Province (Grant number
2023RC3028), and Central South University Innovation-Driven Research
Programme (Grant number 2023CXQD032).
Bao is partially supported by NNSF of China (Grant number 12271328).}
\thanks{Pan Ma is the corresponding author.}

\begin{abstract}
We show that the maximal Fock space $F^\infty_\alpha$ on $\cn$ is a
Lipschitz space, that is, there exists a distance $d_\alpha$ on $\cn$ such that an entire
function $f$ on $\cn$ belongs to $F^\infty_\alpha$ if and only if
$$|f(z)-f(w)|\le Cd_\alpha(z,w)$$
for some constant $C$ and all $z,w\in\cn$. This can be considered the Fock space
version of the following classical result in complex analysis: a holomorphic function
$f$ on the unit ball $\bn$ in $\cn$ belongs to the Bloch space if and only if there exists
a positive constant $C$ such that $|f(z)-f(w)|\le C\beta(z,w)$ for all $z,w\in\bn$,
where $\beta(z,w)$ is the distance on $\bn$ in the Bergman metric. We also present
a new approach to Hardy-Littlewood type characterizations for $F^p_\alpha$.
\end{abstract}

\maketitle

\section{Introduction}

For $\alpha>0$ and $0<p\le\infty$ we use $L^p_\alpha$ or $L^p_\alpha(\cn)$ to denote
the space of all Lebesgue measurable functions $f$ on the complex Euclidean space $\cn$
such that the function $f(z)e^{-\alpha|z|^2/2}$ belongs to $L^p(\cn, dv)$, where $dv$ is
ordinary volume measure on $\cn$. For $f\in L^p_\alpha$ we write
$$\|f\|^p_{p,\alpha}=\left(\frac{p\alpha}{2\pi}\right)^n
\inc\left|f(z)e^{-\alpha|z|^2/2}\right|^p\,dv(z)$$
when $0<p<\infty$ and
$$\|f\|_{\infty,\alpha}=\esssup\left\{|f(z)|e^{-\alpha|z|^2/2}: z\in\cn\right\}$$
when $p=\infty$.

Let $H(\cn)$ denote the space of all entire functions on $\cn$. The spaces
$$F^p_\alpha=L^p_\alpha\cap H(\cn),\qquad 0<p\le\infty, \alpha>0,$$
are usually called Fock spaces. Each $F^p_\alpha$ is closed in the Lebesgue space
$L^p_\alpha$. In particular, $F^p_\alpha$ is a Banach space when $1\le p\le\infty$.

It is clear that $L^2_\alpha=L^2(\cn, \dla)$, where
$$\dla(z)=\left(\frac\alpha\pi\right)^n e^{-\alpha|z|^2}\,dv(z)$$
is the Gaussian measure. The orthogonal projection
$P_\alpha: L^2(\cn, \dla)\to F^2_\alpha$ is an integral operator, namely,
\begin{equation}
P_\alpha f(z)=\inc e^{\alpha z\overline w}f(w)\,\dla(w),
\label{eq1}
\end{equation}
where $z\overline w=z_1\overline w_1+\cdots+z_n\overline w_n$. It is well known
that $F^p_\alpha=P_\alpha L^p_\alpha$ for all $1\le p\le\infty$.

The function $K_w(z)=K(z,w)=e^{\alpha z\overline w}$ is the reproducing kernel of
the Hilbert space $F^2_\alpha$. We will need to use the normalized reproducing kernels
$k_w=K_w/\|K_w\|_{2,\alpha}$, which are unit vectors in $F^2_\alpha$. It is clear that
$$k_w(z)=\frac{K(z,w)}{\sqrt{K(w,w)}}=e^{\alpha z\overline w-(\alpha|w|^2/2)}.$$
See \cite{Zhu2} for an introduction to Fock spaces.

The main result of the paper is Theorem A below.

\begin{theorema}
Suppose $\alpha>0$, $f\in H(\cn)$, and $d$ is the distance function on $\cn$ defined by
$$d(z,w)=\inc\left|e^{\alpha z\overline u}-e^{\alpha w\overline u}\right|
\,d\lambda_{\alpha/2}(u).$$
Then the following conditions are equivalent.
\begin{enumerate}
\item[(a)] $f\in F^\infty_\alpha$.
\item[(b)] There exists a positive constant $C$ such that
$$|f(z)-f(w)|\le Cd(z,w),\qquad z,w\in\cn.$$
\item[(c)] The function $Rf(z)/(1+|z|^2)$ belongs to $L^\infty_\alpha$, where
$$Rf(z)=z_1\partial_1f(z)+\cdots+z_n\partial_nf(z)$$
is the radial derivative of $f$ with $\partial_kf=\partial f/\partial z_k$ for $1\le k\le n$.
\end{enumerate}
\end{theorema}

This result has a well-known analogue in the more classical theory of Bergman spaces.
Recall that the Bergman space $A^p$, $0<p<\infty$, of the open unit ball $\bn$ in
$\cn$ is the space of all holomorphic functions in $L^p(\bn,dv)$. If $P: L^2(\bn, dv)
\to A^2$ is the Bergman projection, then it is well known that $A^p=PL^p(\bn,dv)$
for $1<p<\infty$. When $p=\infty$, the space $\B=PL^\infty(\bn)$ is called the
Bloch space of $\bn$, which can be shown to consist of all holomorphic $f$ on $\bn$
such that
$$\sup_{z\in \bn}(1-|z|^2)|Rf(z)|<\infty.$$
The Bergman space analogue of Theorem A is the following: a holomorphic function
$f$ on $\bn$ belongs to $\B$ if and only if $|f(z)-f(w)|\le C\beta(z,w)$ for some positive
constant $C$ and all $z,w\in\bn$, where $\beta(z,w)$ is the distance function on $\bn$
in the Bergman metric. See \cite{Zhu1}.

The Bergman spaces $A^p$ and the Bloch space $\B$ on $\bn$ can also be described
in terms of higher order derivatives. More precisely, if $f$ is holomorphic on $\bn$ and
$N$ is a positive integer, then $f\in A^p$ if and only if the functions
$(1-|z|)^{|m|}\partial^m f(z)$ belong to $L^p(\bn,dv)$ for all $|m|=N$. Here
$m=(m_1,\cdots,m_n)$ is an $n$-tuple of non-negative integers and
$$|m|=m_1+\cdots+m_n,\qquad\partial^mf=\frac{\partial^{|m|}f}{\partial z_1^{m_1}\cdots\partial z_n^{m_n}}.$$
Similarly, $f\in\B$ if and only if the functions $(1-|z|)^{|m|}\partial^mf(z)$ belong to
$L^\infty(\bn,dv)$ for all $|m|=N$. Such results are usually called
Hardy-Littlewood theorems, especially in the one-dimensional case of the unit disc.
See \cite{Zhu1} again.

It turns out that these Hardy-Littlewood type theorems also hold for Fock spaces. The
following theorem can be found in \cite{CFB, CH, CP, CN, CP1}.

\begin{theoremb}
Suppose $\alpha>0$, $0<p\le\infty$, $N$ is positive integer, and $f\in H(\cn)$.
Then $f\in F^p_\alpha$ if and only if the functions $\partial^mf(z)/(1+|z|)^{|m|}$,
$|m|=N$, all belong to $L^p_\alpha$.
\end{theoremb}

We will present a new approach to Theorem B above when $n=1$. Our proof is based
on the main theorem in \cite{CZ} and is much different and simpler than the existing
proofs in the literature. We will conclude the paper with two open problems.

\section{Some distance functions on $\cn$}

Fix any two positive parameters $\alpha$ and $\beta$ and define a function
$$d_{\alpha,\beta}:\cn\times\cn\to[0,\infty)$$
by
$$d_{\alpha,\beta}(z,w)=\inc\left|e^{\beta z\overline u}-e^{\beta w\overline u}
\right|\,\dla(u).$$
By the ``rotation-invariance'' of the Gaussian measure, we clearly have
$$d_{\alpha,\beta}(z,w)=d_{\alpha,\beta}(Uz, Uw)$$
for all $z,w\in\cn$ and all unitary transformations $U:\cn\to\cn$. The following
lemma is obvious.

\begin{lem}\label{1}
Each $d_{\alpha,\beta}$ is a distance on $\cn$, that is, it satisfies the
following three axioms:
\begin{enumerate}
\item[(a)] $d_{\alpha,\beta}(z,w)\ge0$ for all $z,w\in\cn$, and
$d_{\alpha,\beta}(z,w)=0$ iff $z=w$.
\item[(b)] $d_{\alpha,\beta}(z,w)=d_{\alpha,\beta}(w,z)$ for all $z,w\in\cn$.
\item[(c)] $d_{\alpha,\beta}(z,w)\le d_{\alpha,\beta}(z,u)+d_{\alpha,\beta}(u,w)$
for all $z,w,u\in\cn$.
\end{enumerate}
\end{lem}

The computation of the precise distance $d_{\alpha,\beta}(z,w)$ between $z$ and $w$ is
often difficult. But we have the following estimate when one of the two points is the origin.
Here $F(z)\sim G(z)$ means that there exist positive constants $c$ and $C$ (independent
of $z$ but dependent on other parameters) such that $cF(z)\le G(z)\le CF(z)$ for all
$z\in\cn$.

\begin{prop}\label{2}
For any fixed positive $\alpha$ and $\beta$ we have
$$d_{\alpha,\beta}(z,0)\sim\sqrt{e^{\beta^2|z|^2/(2\alpha)}-1}$$
for $z\in\cn$. In particular,
$$d_{\alpha/2,\alpha}(z,0)\sim\sqrt{e^{\alpha|z|^2}-1},\qquad z\in\cn.$$
\end{prop}

\begin{proof}
It follows from the reproducing property in $F^2_\alpha$ that
$$\inc\left|e^{\beta z\overline u}\right|\,d\lambda_\alpha(u)
=\inc\left|e^{\alpha(\beta z/2\alpha)\overline u}\right|^2\,d\lambda_\alpha(u)
=e^{\alpha|\beta z/2\alpha|^2}=e^{\beta|z|^2/(4\alpha)}.$$
Thus we have the following estimates:
$$d_{\alpha,\beta}(z,0)\le\inc|e^{\beta z\overline u}|\,\dla(u)+1
=e^{\beta^2|z|^2/(4\alpha)}+1,\qquad z\in\cn,$$
and
$$d_{\alpha,\beta}(z,0)\ge\inc|e^{\beta z\overline u}|\,\dla(u)-1
=e^{\beta^2|z|^2/(4\alpha)}-1,\qquad z\in\cn.$$
Consequently,
$$\lim_{|z|\to\infty}\frac{d_{\alpha,\beta}(z,0)}{e^{\beta^2|z|^2/(4\alpha)}}=1,$$
and for any $r>0$ there exists a constant $C>0$ such that
$$C^{-1}e^{\beta^2|z|^2/(4\alpha)}\le d_{\alpha,\beta}(z,0)\le
Ce^{\beta^2|z|^2/(4\alpha)},\qquad |z|\ge r.$$
This clearly implies that there is another positive constant $C$ such that
$$C^{-1}\sqrt{e^{\beta^2|z|^2/(2\alpha)}-1}\le d_{\alpha,\beta}(z,0)
\le C\sqrt{e^{\beta^2|z|^2/(2\alpha)}-1},\qquad |z|\ge r.$$

On the other hand, it is clear from the unitary invariance of the Gaussian measure (or
the distance $d_{\alpha,\beta}$) that we can use the special points $z=(z_1,0,\cdots,0)$
below to obtain
\begin{equation}\label{eq2}
\lim_{|z|\to0}\frac{d_{\alpha,\beta}(z,0)}{|z|}=\lim_{|z|\to0}\inc
\frac{|e^{\beta z\overline u}-1|}{|z|}\,\dla(u)=\beta\inc|u_1|\,\dla(u).
\end{equation}
It follows that
$$d_{\alpha,\beta}(z,0)\sim\sqrt{e^{\beta^2|z|^2/(2\alpha)}-1},\qquad|z|\to0.$$
The proof of the proposition is complete when we combine this with the estimate
at the end of the previous paragraph.
\end{proof}

To simplify notation, we will write
$$d_\alpha(z,w)=d_{\alpha/2,\alpha}(z,w)$$
from this point on. This distance function arises naturally in the study of the spaces
$F^p_\alpha$. For example, we have the following result.

\begin{lem}
For any $\alpha>0$ we have
$$d_\alpha(z,w)\sim\sup\left\{|f(z)-f(w)|: f\in F^\infty_\alpha,\|f\|_{\infty,\alpha}\le1
\right\}$$
for $z,w\in\cn$.
\label{3}
\end{lem}

\begin{proof}
It is well known that the integral operator $P_\alpha$ defined in (\ref{eq1}) maps
$L^\infty_\alpha(\cn)$ boundedly onto $F^\infty_\alpha$. It follows easily that
for $f\in F^\infty_\alpha$ we have
$$\|f\|_{\infty,\alpha}\sim\inf\left\{\|g\|_{\infty,\alpha}:
f=P_\alpha g, \ g\in L^\infty_\alpha(\cn)\right\}.$$

Let $d'_\alpha(z,w)$ denote the supremum in the lemma, which is also a distance on $\cn$
(see \cite{Zhu3} for many other examples of distance functions induced by spaces of
analytic functions). Then
\begin{align*}
d'_\alpha(z,w)&\sim\sup_{\|g\|_{\infty,\alpha}\le1}|P_\alpha g(z)-P_\alpha g(w)|\\
&=\sup_{\|g\|_{\infty,\alpha}\le1}\left|\inc(e^{\alpha z\overline u}
-e^{\alpha w\overline u})g(u)\,\dla(u)\right|\\
&=\inc|e^{\alpha z\overline u}-e^{\alpha w\overline u}|e^{\alpha|z|^2/2}\,\dla(u)\\
&=2^n\inc\left|e^{\alpha z\overline u}-e^{\alpha w\overline u}\right|
\,d\lambda_{\alpha/2}(u).
\end{align*}
This proves the desired estimates.
\end{proof}

It is natural to wonder if the limit in (\ref{eq2}) can be computed at points away from
the origin. When $n=1$, it is clear that
$$\lim_{w\to z}\frac{d_\alpha(z,w)}{|z-w|}=\lim_{w\to z}\int_{\C}
\frac{|e^{\alpha z\overline u}-e^{\alpha w\overline u}|}{|z-w|}\,d\lambda_{\alpha/2}(u)
=\alpha\int_{\C}|u|\left|e^{\alpha z\overline u}\right|\,d\lambda_{\alpha/2}(u).$$
However, the limit above does NOT exist when $n>1$ and $z\not=0$. In fact, if
we choose $w=z+(t,0,\cdots,0)$, then
$$\lim_{t\to0}\inc\frac{|1-e^{\alpha(w-z)\overline u}|}{|w-z|}|e^{\alpha z\overline u}|
\,d\lambda_{\alpha/2}(u)=\alpha\inc|u_1e^{\alpha z\overline u}|
\,d\lambda_{\alpha/2}(u).$$
Other similar ``partial derivatives'' will yield the following sub-limits:
$$\alpha\inc|u_ke^{\alpha z\overline u}|\,d\lambda_{\alpha/2}(u),\qquad 1\le k\le n,$$
which clearly depend on $k$. More specifically, it follows from polar coordinates
and the one-dimensional case that
$$\inc|u_ke^{\alpha z\overline u}|\,d\lambda_{\alpha/2}(u)\sim(1+|z_k|)
e^{\alpha|z|^2/2}.$$

On the other hand, we can write
$$\frac{d_\alpha(z,w)}{|z-w|}=\inc\frac{|1-e^{\alpha(w-z)\overline u}|}
{|(w-z)\overline u|}\,\frac{|(w-z)\overline u|}{|z-w|}\,|e^{\alpha z\overline u}|
\,d\lambda_{\alpha/2}(u).$$
It follows from the Cauchy-Schwarz inequality for vectors in $\cn$ that
\begin{equation}\label{eq3}
\limsup_{w\to z}\frac{d_\alpha(z,w)}{|z-w|}\le\alpha\inc|u||e^{\alpha z\overline u}|
\,d\lambda_{\alpha/2}(u).
\end{equation}
Thus we want to determine the growth rate of the integral
$$E(z)=\inc|u|\,|e^{\alpha z\overline u}|\,d\lambda_{\alpha/2}(u)$$
as $|z|\to\infty$, which will be used several times later on.

\begin{lem}
We have
$$\frac\alpha2\inc\left|u|^2|e^{\alpha z\overline u}\right|\,d\lambda_{\alpha/2}(u)
=\left(n+\frac\alpha2|z|^2\right)e^{\alpha|z|^2/2}$$
for all $z\in\cn$.
\label{4}
\end{lem}

\begin{proof}
It follows from the reproducing property in $F^2_{\alpha/2}$ that
$$\inc\left|e^{\frac\alpha2z\overline u}\right|^2\,d\lambda_{\alpha/2}(u)
=e^{\frac\alpha2|z|^2}$$
for all $z\in\cn$. Apply $\partial^2/\partial z_k\partial\overline z_k$ to both sides of
the above identity. The result is
$$\left(\frac\alpha2\right)^2\inc|u_k|^2\left|e^{\frac\alpha2z\overline u}\right|^2
\,d\lambda_{\alpha/2}(u)=\frac\alpha2\left(1+\frac\alpha2|z_k|^2\right)
e^{\frac\alpha2|z|^2}.$$
Summing over $k$, we obtain
$$\frac\alpha2\inc|u|^2\left|e^{\alpha z\overline u}\right|\,d\lambda_{\alpha/2}(u)
=\left(n+\frac\alpha2|z|^2\right)e^{\frac\alpha2|z|^2},$$
completing the proof of the lemma.
\end{proof}

\begin{lem}
For any $\alpha>0$ we have $E(z)\sim(1+|z|)e^{\alpha|z|^2/2}$ on $\cn$.
\label{5}
\end{lem}

\begin{proof}
We write $z=rw$, where $r=|z|$ (so $|w|=1$). Then
$$\inc e^{\alpha rw\overline u/2}e^{\alpha r\overline wu/2}\,d\lambda_{\alpha/2}(u)
=\inc\left|e^{\alpha z\overline u/2}\right|^2\,d\lambda_{\alpha/2}(u)
=e^{\alpha|z|^2/2}=e^{\alpha r^2/2}.$$
Take the derivative with respect to $r$ on both sides. We obtain
$$\frac\alpha2\inc(\overline wu+w\overline u)\left|e^{\alpha r\overline u}\right|
\,d\lambda_{\alpha/2}(u)=\alpha re^{\alpha r^2/2},$$
or
$$\inc\re(w\overline u)\left|e^{\alpha z\overline u}\right|\,d\lambda_{\alpha/2}(u)
=|z|e^{\alpha|z|^2/2}.$$
Therefore,
$$|z|e^{\alpha|z|^2/2}\le\inc|u|\left|e^{\alpha z\overline u}\right|
\,d\lambda_{\alpha/2}(u).$$
It is clear that the integral on the right-hand side above is a strictly positive continuous
function of $z$ for $|z|\le1$, so there is a positive constant $c$ such that
$$c\,(1+|z|)e^{\alpha|z|^2/2}\le\inc|u|\left|e^{\alpha z\overline u}\right|
\,d\lambda_{\alpha/2}(u)$$
for all $z\in\cn$.

On the other hand, it follows from H\"older's inequality and Lemma~\ref{4} that
\begin{align*}
\left[\inc|u|\left|e^{\alpha z\overline u}\right|\,d\lambda_{\alpha/2}(u)\right]^2
&\le\inc|u|^2\left|e^{\alpha z\overline u}\right|\,d\lambda_{\alpha/2}(u)
\inc\left|e^{\alpha z\overline u}\right|\,d\lambda_{\alpha/2}(u)\\
&=\left(\frac{2n}\alpha+|z|^2\right)e^{\frac\alpha2|z|^2}e^{\frac\alpha2|z|^2}\\
&=\left(\frac{2n}\alpha+|z|^2\right)e^{\alpha|z|^2}.
\end{align*}
Thus there exists a positive constant $C=C_\alpha$ such that
$$\inc|u|\left|e^{\alpha z\overline u}\right|\,d\lambda_{\alpha/2}(u)
\le C(1+|z|)\,e^{\frac\alpha2|z|^2}$$
for all $z\in\C$. Combining this with the estimate in the previous paragraph,
we complete the proof of the lemma.
\end{proof}

Finally in this section we note that for any $\alpha>0$, $\beta>0$, and $p>1$, the function
$$d(z,w)=\left[\inc|e^{\beta z\overline u}-e^{\beta w\overline u}|^p
\,\dla(u)\right]^{\frac1p}$$
is also a distance on $\cn$. For $0<p<1$ the following is a distance on $\cn$:
$$d(z,w)=\inc|e^{\beta z\overline u}-e^{\beta w\overline u}|^p\,\dla(u).$$
It is not clear what these more general distances might be good for. It would be nice
to find some applications for them.

\section{Characterizations of $F^\infty_\alpha$}

In this section we prove two characterizations for the space $F^\infty_\alpha$. It is well
known that $F^\infty_\alpha$ is maximal in some sense among Banach spaces of entire
functions under the action of the Heisenberg group; see \cite{Zhu3}. Our first
characterization of $F^\infty_\alpha$ below shows that it is a Lipschitz space.

\begin{thm}
Let $\alpha>0$ and $f\in H(\cn)$. Then $f\in F^\infty_\alpha$ if and only if
there exists a positive constant $C$ such that
$$|f(z)-f(w)|\le Cd_\alpha(z,w)$$
for all $z,w\in\cn$.
\label{6}
\end{thm}

\begin{proof}
The ``only if'' direction is a direct consequence of Lemma~\ref{3}. Alternatively,
for any $f\in F^\infty_\alpha$ we have
\begin{align*}
|f(z)-f(w)|&=\left|\inc(e^{\alpha z\overline u}-e^{\alpha w\overline u})f(u)
\,d\lambda_\alpha(u)\right|\\
&\le\left(\frac\alpha\pi\right)^n\inc|e^{\alpha z\overline u}-e^{\alpha w\overline u}|
|f(u)e^{-\alpha|u|^2/2}|e^{-\alpha|u|^2/2}\,dv(u)\\
&\le2^n\|f\|_{\infty,\alpha}\inc|e^{\alpha z\overline u}-e^{\alpha w\overline u}|\,
d\lambda_{\alpha/2}(u)\\
&=2^n\|f\|_{\infty,\alpha}d_\alpha(z,w)
\end{align*}
for all $f\in F^\infty_\alpha$ and $z,w\in\cn$.

On the other hand, if $f$ satisfies the Lipschitz condition, then in particular,
$|f(z)-f(0)|\le Cd_\alpha(z,0)$ for all $z\in\cn$. By Proposition~\ref{2}, there is another
positive constant $C$ such that $|f(z)-f(0)|\le C\sqrt{e^{\alpha|z|^2}-1}$ for all
$z\in\cn$, which clearly implies that the function $f(z)e^{-\alpha|z|^2/2}$ is bounded
on $\cn$. This completes the proof of the theorem.
\end{proof}

Again, to put the result above in proper perspective, we should think of it as the Fock
space version of the following well-known result for holomorphic functions $f$ on the
unit ball $\bn$: $f$ belongs to the Bloch space $\B$ if and only if there exists a constant
$C$ such that $|f(z)-f(w)|\le C\beta(z,w)$ for all $z,w\in\bn$, where $\beta(z,w)$
is the distance between $z$ and $w$ in the Bergman metric. See \cite{Zhu1}.

For a function $f\in H(\cn)$ we will write $\nabla f=(\partial_1f,\cdots,\partial_nf)$
for the holomorphic gradient of $f$. The equivalence of conditions (a), (b), and (c)
below is known, and we include a simple proof here. But condition (d) appears to be
new, interesting, and non-trivial.

\begin{thm}
Suppose $\alpha>0$ and $f$ is an entire function on $\cn$. Then the following conditions
are equivalent:
\begin{enumerate}
\item[(a)] $f\in F^\infty_\alpha$.
\item[(b)] There exists a positive constant $C$ such that
$$|\partial_kf(z)|\le C(1+|z|)e^{\alpha|z|^2/2}$$
for all $z\in\cn$ and $1\le k\le n$.
\item[(c)] There exists a positive constant $C$ such that
$$|\nabla f(z)|\le C(1+|z|)e^{\alpha|z|^2/2}$$
for all $z\in\cn$.
\item[(d)] There exists a positive constant $C$ such that
$$|Rf(z)|\le C(1+|z|^2)e^{\alpha|z|^2/2}$$
for all $z\in\cn$.
\end{enumerate}
\label{7}
\end{thm}

\begin{proof}
It is obvious that conditions (b) and (c) are equivalent.

If $f\in F^\infty_\alpha$, then by Theorem~\ref{6}, there exists a positive
constant $C$ such that
$$|f(z)-f(w)|\le C\inc\left|e^{\alpha z\overline u}
-e^{\alpha w\overline u}\right|\,d\lambda_{\alpha/2}(u)$$
for all $z$ and $w$ in $\C$. Let $w=z+(h,0,\cdots,0)$, divide both sides of the inequality
above by $|h|$, and then let $h\to0$. The result is
$$|\partial_1f(z)|\le C\alpha\inc\left|u_1e^{\alpha z\overline u}\right|
\,d\lambda_{\alpha/2}(u).$$
It is clear that we also have
$$|\partial_kf(z)|\le C\alpha\inc\left|u_ke^{\alpha z\overline u}\right|
\,d\lambda_{\alpha/2}(u)$$
for all $1\le k\le n$. By Lemma \ref{5}, there exists another constant $C$ such that
$$|\partial_kf(z)|\le C(1+|z|)e^{\alpha|z|^2/2}$$
for all $1\le k\le n$ and $z\in\C$. Thus condition (a) implies (b).

By the Cauchy-Schwarz inequality, we have $|Rf(z)|\le|z||\nabla f(z)|$, which yields
\begin{equation}\label{eq4}
\frac{|Rf(z)|}{1+|z|^2}\le\frac{|z||\nabla f(z)|}{1+|z|^2}\le\frac{2|\nabla f(z)|}{1+|z|}
\end{equation}
for all $z\in\cn$. This shows that condition (c) implies (d).

Finally, we assume that condition (d) holds. For any $z\in\cn$ we can write
\begin{align}\label{eq5}
f(z)-f(0)&=\sum_{k=1}^nz_k\int_0^1\partial_kf(tz)\,dt=\int_0^1Rf(tz)\,\frac{dt}t\\
&=\int_0^{1/2}Rf(tz)\,\frac{dt}t+\int_{1/2}^1Rf(tz)\,\frac{dt}t.\nonumber
\end{align}
Let $I_1(z)$ and $I_2(z)$ denote the two integrals above, respectively. Then
\begin{align*}
|I_2(z)|&\le\int_{1/2}^1|Rf(tz)|\,\frac{dt}t\le2\int_{1/2}^1|Rf(tz)|\,dt\\
&\le2C\int_0^1(1+t^2|z|^2)e^{\alpha t^2|z|^2/2}\,dt\\
&\le2C\left[e^{\alpha|z|^2/2}+\int_0^{|z|}se^{\alpha s^2/2}\,ds\right]\\
&=2C\left[e^{\alpha|z|^2/2}+\frac1\alpha\left(e^{\alpha|z|^2/2}-1\right)\right]\\
&\sim e^{\alpha|z|^2/2},\qquad z\in\cn.
\end{align*}
To estimate $I_1(z)$, note that the assumption
$$|Rf(z)|\le C(1+|z|^2)e^{\alpha|z|^2/2}$$
is clearly equivalent to
$$|Rf(z)|\le C|z|(1+|z|)e^{\alpha|z|^2/2}$$
(with a possibly different constant). Thus
\begin{align*}
|I_1(z)|&\le\int_0^{1/2}|Rf(tz)|\,\frac{dt}t\\
&\le C\int_0^{1/2}|z|(1+t|z|)e^{\alpha t^2|z|^2/2}\,dt\\
&\le \frac C2|z|\left(1+\frac{|z|}2\right)e^{\alpha|z|^2/8}\\
&\le C'e^{\alpha|z|^2/2},\qquad z\in\cn
\end{align*}
for another positive constant $C'$. Combining the estimates for $I_1(z)$ and $I_2(z)$,
we find another positive constant $C$ such that
$$|f(z)-f(0)|\le Ce^{\alpha|z|^2/2}$$
for all $z\in\cn$. This shows $f\in F^\infty_\alpha$ and completes the proof
of the theorem.
\end{proof}

We warn any inexperienced reader that the factor $1-|z|^2$ on the unit ball
$\bn$ can be replaced by $1-|z|$, while $1+|z|^2$ is critically different from $1+|z|$
in Theorem~\ref{7} above!

\section{Hardy-Littlewood type theorems for Fock spaces}

It is known that Theorem \ref{7} can be extended to all Fock spaces $F^p_\alpha$
in terms of higher order derivatives.

\begin{thm}
Suppose $0<p\le\infty$, $\alpha>0$, $f\in H(\cn)$, and $N$ is a positive integer.
Then $f\in F^p_\alpha$ if and only if the functions $\partial^mf(z)/(1+|z|)^m$,
where $|m|=N$, all belong to $L^p_\alpha$.
\label{8}
\end{thm}

\begin{proof}
See \cite{CFB, CH, CP, CN, CP1}.
\end{proof}

In this section, we provide a new approach to Theorem~\ref{8} in the one-dimensional
case. This approach is based on the main theorem in \cite{CZ} and is much easier than
the arguments used in other papers in the literature. We begin with the case of first
order derivatives.

\begin{thm}
Suppose $f$ is an entire function on $\C$, $\alpha>0$, and $0<p\leq\infty$. Then
$f\in F^p_\alpha$ if and only if the function $f'(z)/(1+|z|)$ belongs to $L^p_\alpha$.
\label{9}
\end{thm}

\begin{proof}
We will prove the result with the help of \cite{CZ}, where it was proved that
$f'\in F^p_\alpha$ if and only if the function $zf(z)$ is in $F^p_\alpha$. What we want
to prove here is that $f\in F^p_\alpha$ if and only if the function $[f'(z)-f'(0)]/z$ is in
$F^p_\alpha$.

Without loss of generality, we may assume that $f(0)=f'(0)=0$. In the equivalence
$$zf(z)\in F^p_\alpha\Longleftrightarrow  f'(z)\in F^p_\alpha,$$
if we replace $f(z)$ by $f(z)/z$, then
$$f(z)\in F^p_\alpha\Longleftrightarrow\frac{f'(z)}z-\frac{f(z)}{z^2}\in F^p_\alpha.$$
Since $f(z)\in F^p_\alpha$ clearly implies that $f(z)/z^2\in F^p_\alpha$, it follows that
$$f(z)\in F^p_\alpha\Longrightarrow f'(z)/z\in F^p_\alpha\Longleftrightarrow
f'(z)/(1+|z|)\in L^p_\alpha.$$

On the other hand, if $f'(z)/z\in F^p_\alpha\subset F^\infty_\alpha$, then by what we
have proved about $F^\infty_\alpha$, we must have $f\in F^\infty_\alpha$, which implies
that $f(z)/z^2\in F^p_\alpha$ when $p>1/2$. Thus $f'(z)/z\in F^p_\alpha$
implies $f'(z)/z-f(z)/z^2\in F^p_\alpha$, which yields $f\in F^p_\alpha$
for $p>1/2$.

For $0<p\le1/2$ (actually, the argument below works for $p>1/2$ as well), we write
$$\frac{f'(z)}z=\left[\frac{f'(z)}z-\frac{f(z)}{z^2}\right]+\frac{f(z)}{z^2}.$$
By the triangle inequality for the distance
$$d(f_1,f_2)=\intc|f_1(z)-f_2(z)|^p\,\dla(z)$$
in $F^p_\alpha$ (technically, we should first work with $f_r(z)=f(rz)$, $0<r<1$, and
then use an approximation argument in order to make sure that all the integrals below
all converge) and one of the main results in \cite{CZ}, there exists a positive constant
$c$ such that
\begin{align*}
\intc\left|\frac{f'(z)}{z}\right|^p\,\dla(z)&\ge
\intc\left|\frac{f'(z)}{z}-\frac{f(z)}{z^2}\right|^p\,\dla(z)-
\intc\left|\frac{f(z)}{z^2}\right|^p\,\dla(z)\\
&\ge c\intc|f(z)|^p\,\dla(z)-
\intc\left|\frac{f(z)}{z^2}\right|^p\,\dla(z).
\end{align*}
Choose a positive radius $R$ such that $c-1/R^{2p}>0$. Then
\begin{multline*}
\intc\left|\frac{f'(z)}z\right|^p\,\dla(z)
\ge\left(c-\frac1{R^{2p}}\right)\int_{|z|>R}|f(z)|^p\,\dla(z)\\
+c\int_{|z|\le R}|f(z)|^p\,\dla(z)-\int_{|z|\le R}\left|\frac{f(z)}{z^2}\right|^p\,\dla(z)
\end{multline*}
whenever $f(0)=f'(0)=0$. This shows that
$$\int_{|z|>R}|f(z)|^p\,\dla(z)<\infty$$
and hence
$$\intc|f(z)|^p\,\dla(z)<\infty$$
if the function $f'(z)/z$ belongs to $F^p_\alpha$.
\end{proof}

\begin{cor}
Suppose $\alpha>0$, $0<p\le\infty$, and $f\in H(\C)$ with $f(0)=f'(0)=0$.
Then for any constant $c$ (including $c=0$) we have $f\in F^p_\alpha$
if and only if the function
$$\frac{f'(z)}z+c\,\frac{f(z)}{z^2}$$
belongs to $F^p_\alpha$.
\label{10}
\end{cor}

\begin{proof}
This follows from Theorem~\ref{9} and its proof.
\end{proof}

\begin{thm}
Suppose $f$ is an entire function on $\C$, $N$ is a positive integer, and $0<p\le\infty$.
Then $f\in F^p_\alpha$ if and only if the functions $f^{(N)}(z)/(1+|z|)^N$ belongs to $L^p_\alpha$.
\label{11}
\end{thm}

\begin{proof}
We prove this by induction on $N$. The case $N=1$ has already been proved.

Suppose the result holds for some positive integer $N$. We proceed to show that
$f\in F^p_\alpha$ if and only if the function $f^{(N+1)}(z)/(1+|z|^{N+1})$ belongs to
$L^p_\alpha$. Without loss of generality. we may assume that
$$f(0)=f'(0)=\cdots=f^{(2N+2)}(0)=0.$$
In this case, we just need to show that $f\in F^p_\alpha$ if and only if the function
$f^{(N+1)}(z)/z^{N+1}$ belongs to $L^p_\alpha$.

By the induction hypothesis, we have that $f\in F^p_\alpha$ if and only if the function
$f^{(N)}(z)/(1+|z|)^N$ belongs to $L^p_\alpha$, which is the same as the function
$g(z)=f^{(N)}(z)/z^N$ belonging to $F^p_\alpha$. By Corollary~\ref{10}, this is
equivalent to the function
$$\frac{g'(z)}z+N\,\frac{g(z)}{z^2}=\frac{f^{(N+1)}(z)}{z^{N+1}}
-\frac{Nf^{(N)}(z)}{z^{N+2}}+\frac{Nf^{(N)}(z)}{z^{N+2}}
=\frac{f^{(N+1)}(z)}{z^{N+1}}$$
belonging to $F^p_\alpha$. So the desired result is true for $N+1$, and the proof of
the theorem is complete.
\end{proof}

Serious obstacles arise when we try the arguments above in higher dimensions,
although several steps still work. In particular, the ``only if part'' of
Theorem~\ref{9} in the higher dimensional case follows easily from the one-dimensional
case. In fact, by Theorem \ref{9} and the closed-graph theorem, there exists a positive
constant $C$ (independent of $f$) such that
$$\int_{\C}\left|\frac{f'(z)e^{-\alpha|z|^2/2}}{1+|z|}\right|^p\,dA(z)
\le C\int_{\C}\left|f(z)e^{-\alpha|z|^2/2}\right|^p\,dA(z)$$
for all $f\in H(\C)$, where $dA$ is ordinary area measure on $\C$.
Now if $f\in H(\cn)$ and $1\le k\le n$, then
\begin{multline*}
\int_{\C}\left|\frac{\partial_kf(z_1,\cdots,z_k,\cdots,z_n)e^{-\alpha|z_k|^2/2}}
{1+|z_k|}\right|^p\,dA(z_k)\\
\le C\int_{\C}\left|f(z_1,\cdots,z_k,\cdots,z_n)e^{-\alpha|z_k|^2/2}\right|^p\,dA(z_k),
\end{multline*}
where $C$ is independent of the $n-1$ variables $\{z_1,\cdots,z_n\}\setminus\{z_k\}$.
Since $1/(1+|z|)\le1/(1+|z_k|)$ for $1\le k\le n$, $|z|^2=|z_1|^2+\cdots+|z_n|^2$,
and the Gaussian measure on $\cn$ is a product measure, we easily deduce that
$$\inc\left|\frac{\partial_kf(z)e^{-\alpha|z|^2/2}}{1+|z|}\right|^p\,dv(z)\le
C\inc\left|f(z)e^{-\alpha|z|^2/2}\right|^p\,dv(z).$$

It then follows that there exists another positive constant $C$ such that
$$\inc\left[\frac{|\nabla f(z)|}{1+|z|}\,e^{-\alpha|z|^2}\right]^p\,dv(z)
\le C\inc\left|f(z)e^{-\alpha|z|^2/2}\right|^p\,dv(z)$$
for all $f\in H(\cn)$. Since $|Rf(z)|\le|z||\nabla f(z)|$, we can also find a
positive constant $C$ such that
$$\inc\left|\frac{Rf(z)}{1+|z|^2}\,e^{-\alpha|z|^2/2}\right|^p\,dv(z)
\le C\inc\left|f(z)e^{-\alpha|z|^2/2}\right|^p\,dv(z)$$
for all $f\in F^p_\alpha$.

When $p=1$, the other direction of the inequalities above can also be proved using
elementary arguments. In fact, it follows from (\ref{eq5}) and Fubini's theorem that
\begin{align*}
\inc|f(z)-f(0)|e^{-\beta|z|^2}\,dv(z)&\le\inc e^{-\beta|z|^2}\,dv(z)\int_0^1
\frac{|Rf(tz)|}t\,dt\\
&=\int_0^1\frac{dt}t\inc|Rf(tz)|e^{-\beta|z|^2}\,dv(z)\\
&=\int_0^1\frac{dt}{t^{2n+1}}\inc|Rf(z)|e^{-\beta|z|^2/t^2}\,dv(z)\\
&=\inc|Rf(z)|\,dv(z)\int_0^1\frac{e^{-\beta|z|^2/t^2}}{t^{2n+1}}\,dt\\
&=\frac1{2\beta^n}\inc\frac{|Rf(z)|}{|z|^{2n}}\,dv(z)\int_{\beta|z|^2}^\infty
s^{n-1}e^{-s}\,ds.
\end{align*}

An argument using mathematical induction shows that the incomplete gamma function
$$\Gamma(n,x)=\int_x^\infty s^{n-1}e^{-s}\,ds,\qquad x\in(0,\infty),$$
has the property that $\Gamma(n,x)\sim x^{n-1}e^{-x}$ as $x\to\infty$, where $n$
is any positive integer. It follows that there exists a positive constant $C$ such that
$$\inc|f(z)-f(0)|e^{-\beta|z|^2}\,dv(z)\le C\inc\frac{|Rf(z)|}{1+|z|^2}e^{-\beta|z|^2}
\,dv(z).$$
This together with (\ref{eq4}) shows that we also have
$$\inc|f(z)-f(0)|e^{-\beta|z|^2}\,dv(z)\le C\inc\frac{|\nabla f(z)|}{1+|z|}
e^{-\beta|z|^2}\,dv(z),$$
where the positive constant $C$ only depends on $n$ and $\beta$.

\section{Further remarks}

It follows from the analysis in previous sections that we have the following
results about Fock spaces in terms of the radial derivative.

\begin{cor}\label{12}
Suppose $f\in H(\cn)$ and $\alpha>0$.
\begin{enumerate}
\item[(a)] If $0<p\le\infty$ and $f\in F^p_\alpha$, then the function
$Rf(z)/(1+|z|^2)$ belongs to $L^p_\alpha$.
\item[(b)] If $p=1$ or $p=\infty$, and if the function $Rf(z)/(1+|z|^2)$ belongs to
$L^p_\alpha$, then $f\in F^p_\alpha$.
\end{enumerate}
\end{cor}

It is therefore very natural for us to make the following conjecture.

\begin{conj}\label{13}
Suppose $0<p\le\infty$, $\alpha>0$, and $f\in H(\cn)$. Then $f\in F^p_\alpha$ if and only
if the function $Rf(z)/(1+|z|^2)$ belongs to $L^p_\alpha$. More generally, if $N$ is any
positive integer, then $f\in F^p_\alpha$ if and only if the function $R^Nf(z)/(1+|z|^2)^N$
belongs to $L^p_\alpha$.
\end{conj}

The ``only if'' parts above follows from Theorem~\ref{8} and the expression of $R^Nf$
in terms of partial derivatives. For example, if $N=2$, we have
\begin{align*}
R^2f&=R(z_1\partial_1f+\cdots+z_n\partial_nf)\\
&=\sum_{k=1}^nR(z_k\partial_kf)=\sum_{k=1}^n\sum_{j=1}^n
z_j\partial_j(z_k\partial_kf)\\
&=\sum_{k=1}^n\left[z_k\partial_kf+\sum_{j=1}^nz_jz_k
\frac{\partial^2f}{\partial z_j\partial z_k}\right]\\
&=Rf+\sum_{j,k=1}^nz_jz_k\frac{\partial^2f}{\partial z_j\partial z_k}.
\end{align*}
Similar formulas can be obtained for $R^Nf$ when $N$ is any positive integer.

It is clear from the previous sections that, for each $\alpha>0$, the distance function
$d_\alpha(z,w)$ plays a signficant role in the study of the Fock spaces $F^p_\alpha$.
However, we have very limited information about these distance functions.

Proposition~\ref{2} gives a good estimate for $d_\alpha(0,z)$. A natural question is
whether or not we can use the estimate in Proposition~\ref{2} together with Weyl
unitary operators (see \cite{Zhu2}) to obtain optimal estimates for the distance
function $d_\alpha(z,w)$. Our attempts so far have been unsuccessful.

Recall that Bergman spaces $A^p$ can be characterized by Lipschitz type conditions
$$|f(z)-f(w)|\le\beta(z,w)\left[g(z)+g(w)\right],$$
where $g\in L^p(\bn, dv)$ and $\beta(z,w)$ is the Bergman distance between
$z$ and $w$. See \cite{WZ}. It is natural to ask whether or not something similar is also
true for Fock spaces. We make the following conjecture here.

\begin{conj}\label{14}
Suppose $\alpha>0$, $0<p\le\infty$, and $f$ is an entire function on $\cn$. Then
$f\in F^p_\alpha$ if and only if there exists a non-negative
continuous function $g\in L^p(\cn, dv)$ such that
\begin{equation}
|f(z)-f(w)|\le d_\alpha(z,w)\left[g(z)+g(w)\right]
\label{eq6}
\end{equation}
for all $z,w\in\cn$.
\end{conj}

If $f$ satisfies the Lipschitz type condition in (\ref{eq6}), then
$$\frac{|f(z)-f(w)|}{|z-w|}\le \frac{d_\alpha(z,w)}{|z-w|}\left[g(z)+g(w)\right]$$
for all $z\not=w$ in $\cn$. Fix $z$, let $w\to z$, and use (\ref{eq3}) and Lemma~\ref{5}.
We obtain a positive constant $C$ such that
$$|\partial_kf(z)|\le C(1+|z|)e^{\alpha|z|^2/2}g(z),\qquad 1\le k\le n,z\in\cn.$$
It follows that the functions $\partial_kf(z)/(1+|z|)$, $1\le k\le n$, all belong
to $L^p_\alpha$. By Theorem \ref{8}, we have $f\in F^p_\alpha$.

To prove the other direction, it seems that we need more detailed information
and more properties of the distance function $d_\alpha(z,w)$, which are not available
at this point. We intend to pursue these issues in a future paper.

\end{document}